\documentclass[11pt]{amsart}
\usepackage{amsmath, amsfonts, amsthm, amssymb, multicol}
\usepackage{graphicx, xcolor}
\usepackage{float}
\usepackage{verbatim}

\allowdisplaybreaks

\usepackage[square,sort,comma,numbers]{natbib}
\usepackage{fancyhdr}
 
\numberwithin{equation}{section}

\newtheorem{thm}{Theorem}[section]

\newtheorem{ques}[thm]{Question}

\renewcommand{\epsilon}{\varepsilon}

\renewcommand{\ge}{\geqslant}
\renewcommand{\le}{\leqslant}
\renewcommand{\geq}{\geqslant}
\renewcommand{\leq}{\leqslant}

\DeclareMathOperator{\spt}{spt}
\DeclareMathOperator{\supp}{spt}

\newcommand{\dd}{\,\mathrm{d}}
\newcommand{\dimh}{\dim_\mathrm{H}}
\newcommand{\dimf}{\dim_\mathrm{F}}

\newtheorem{theorem}{Theorem}[section]
\newtheorem{corollary}[theorem]{Corollary}
\newtheorem{proposition}[theorem]{Proposition}

\theoremstyle{definition}

\title{On the Fourier dimension of $(d,k)$-sets and Kakeya  sets with restricted directions}

\author{Jonathan M. Fraser$^1$, Terence L.~J.~Harris$^2$ \& Nicholas G. Kroon$^1$\\ \\
 $^1$S\MakeLowercase{chool of} M\MakeLowercase{athematics and} S\MakeLowercase{tatistics}, U\MakeLowercase{niversity of} S\MakeLowercase{t} A\MakeLowercase{ndrews}, S\MakeLowercase{cotland} \\
$^2$D\MakeLowercase{epartment of} M\MakeLowercase{athematics}, C\MakeLowercase{ornell} U\MakeLowercase{niversity}, I\MakeLowercase{thaca}, NY 14853, USA\\
\MakeLowercase{Emails: jmf32@st-andrews.ac.uk, tlh236@cornell.edu \& nkroon3@gmail.com}}
\thanks{JMF was  financially supported by an \emph{EPSRC Standard Grant} (EP/R015104/1) and a  \emph{Leverhulme Trust Research Project Grant} (RPG-2019-034). }

\begin{document}


\maketitle
\thispagestyle{empty}

\begin{abstract}
A $(d,k)$-set is a  subset of $\mathbb{R}^d$ containing a $k$-dimensional unit ball  of all possible orientations.    Using an approach of D.~Oberlin we prove various Fourier dimension estimates for compact $(d,k)$-sets.  Our main interest is in restricted $(d,k)$-sets, where the set only contains unit balls with a restricted set of possible orientations $\Gamma$.  In this setting our estimates depend on the Hausdorff dimension of $\Gamma$ and can sometimes be improved if additional geometric properties of $\Gamma$ are assumed. We are led to consider cones and prove that the cone in $\mathbb{R}^{d+1}$ has Fourier dimension $d-1$, which may be of interest in its own right.
\end{abstract}

\section{Kakeya sets, $(d,k)$-sets, and dimension theory}

A \textit{Kakeya set} is a subset of $\mathbb{R}^d$ containing a unit line segment in every direction.   Besicovitch \cite{besicovitch} proved that there exist Kakeya sets with zero $d$-dimensional Lebesgue measure (for any $d \geq 2$) and it is a notorious  problem in geometric measure theory and harmonic analysis to determine if Kakeya sets can be even smaller than this, that is, can they have Hausdorff dimension strictly less than $d$?  The case $d=1$ is trivial and we assume throughout that $d \geq 2$.  Davies proved that Kakeya sets in $\mathbb{R}^2$ must have Hausdorff dimension 2 \cite{davies}.  The problem   is open for $d \geq 3$ but various partial results are known.  Notably, it was proved in \cite{katz}   that Kakeya  sets in $\mathbb{R}^3$ have Hausdorff dimension at least  $5/2+\varepsilon$ for some small constant $\varepsilon>0$ and, in the general case, it was proved in \cite{katztao}   that Kakeya  sets have Hausdorff dimension at least $(2-\sqrt{2})(d-4)+3$. Further improvements in certain dimensions were achieved recently in  \cite{hickman}.  See \cite[Figure 5]{hickman} for a survey of the state of the art.

Oberlin gave a Fourier analytic proof of Davies' result \cite{oberlin} which actually establishes something stronger:  a compact Kakeya set in $\mathbb{R}^2$ must have \emph{Fourier} dimension 2.  Oberlin's result and proof are the starting point for our work and we use his general approach to study variants of the Kakeya problem, especially $(d,k)$-sets with restricted orientations.  

A $(d,k)$-set is a  subset of $\mathbb{R}^d$ containing a $k$-dimensional unit ball  of all possible orientations.  As such, $(d,1)$-sets are  Kakeya sets.  We give a more formal definition below. One can now ask if $(d,k)$-sets with zero $d$-dimensional Lebesgue measure exist for  $d > k \geq 2$?  In fact, this is an open problem in general but it is conjectured that no such sets exist.  Falconer \cite{falconer_1980} proved that $(d,k)$-sets have positive measure whenever $k> d/2$. Around the same time Marstrand \cite{marstrand} proved  that $(3,2)$-sets have positive 3-dimensional measure using a different approach.   This result has subsequently been strengthened by Bourgain \cite{bourgain} and R. Oberlin \cite{oberlin10}. As far as we know the state of the art is that $(d,k)$-sets necessarily have positive Lebesgue measure when $(1+\sqrt{2})^{k-1}+k>d$, see \cite[Chapter 24]{mattila_2015} and the survey \cite{mattila_2019}.

Our main interest is in restricted $(d,k)$-sets, which we introduce now, and in estimates for the Fourier dimension.  The Grassmannian manifold $G(d,k)$ consists of all $k$-dimensional subspaces of $\mathbb{R}^d$.   This is a smooth compact manifold of dimension $k(d-k)$, see \cite{mattila_1995}.  To formally define $(d,k)$-sets it is convenient to associate each subspace  $s\in G(d,k)$ with an orthonormal  basis $\{x_1^s,x_2^s,\ldots,x_k^s\} \subseteq \mathbb{R}^d$.   In what follows it should be clear that the specific choice of basis is irrelevant.  Let $\Gamma \subseteq G(d,k)$.  If  $E\subseteq\mathbb{R}^d$ is a $(d,k, \Gamma)$\textit{-set}, then  for all  $s\in \Gamma$ there exists a translation  $t_s\in \mathbb{R}^d$ such that \[
t_s+\sum_{i=1}^kr_ix_i^s\in E
\]
 for all $r=(r_1,r_2,\ldots,r_k)\in[0,1]^k$.   In particular, a $(d,k)$-set is a $(d,k,G(d,k))$-set.  We aim to bound the Fourier dimension of compact $(d,k, \Gamma)$-sets in terms of $d,k$, and geometric properties of $\Gamma$.  The Fourier dimension of a set is bounded above by the Hausdorff dimension, hence lower bounds for the Fourier dimension give lower bounds for Hausdorff dimension.   A related problem was considered by Oberlin \cite{oberlin_israel}.  This paper considers sets in $\mathbb{R}^d$ containing certain families of affine hyperplanes.  If the Hausdorff dimension of the family is large enough, then it is proved in \cite[Theorem 1.3]{oberlin_israel} that the Lebesgue measure of the set must be positive.

The \textit{Fourier transform} of a Lebesgue integrable, complex-valued   function $f$ on $\mathbb{R}^d$   is the function $\hat{f} : \mathbb{R}^d \to \mathbb{C}$ given by
\[\hat{f}(\xi) = \int_{\mathbb{R}^d} f(x) e^{-2\pi i\xi\cdot x}\dd x.\]
Analogously, the \textit{Fourier transform} of a Borel  measure $\mu$ on $\mathbb{R}^d$ is the function $\hat{\mu} : \mathbb{R}^d \to \mathbb{C}$ given by
\[\hat{\mu}(\xi) = \int_{\mathbb{R}^d}   e^{-2\pi i\xi\cdot x} \dd \mu(x).\]

We write $\mathcal{M}(E)$ to denote the set of all Borel probability measures supported on a closed set $E$.  Throughout we write $A \lesssim B$ to mean there is a constant $c >0$ such that $A \leq cB$.   If the implicit constant $c$ depends on another parameter $\varepsilon$ we write $A \lesssim_\varepsilon B$.  The \textit{Fourier dimension} of a closed set $E\subseteq\mathbb{R}^d$ is then 
\[\dimf E =\sup\{0 \leq s\le d:\exists\, \mu\in\mathcal{M}(E)\text{ such that }|\hat{\mu}(\xi)| \lesssim_s |\xi|^{-s/2} \}.\]
By relating energy to the Fourier transform, it is straightforward to see that $\dimf E \leq \dimh E$, see \cite{mattila_2015}, where $\dimh$ denotes Hausdorff dimension.

There exist  sets in $\mathbb{R}^d$ with Fourier dimension $d$ but zero $d$-dimensional  Lebesgue measure. For example, Besicovitch sets in $\mathbb{R}^2$.  However, if we know that the Fourier decay of a measure $\mu\in\mathcal{M}(E)$ is ``even better'' than is required for full dimension, i.e. if there exists $\varepsilon>0$ such that
\[|\hat{\mu}(\xi)| \lesssim_\varepsilon  |\xi|^{-(d+\varepsilon)/2} \]
then $E \subseteq \mathbb{R}^d$ not only has full Fourier dimension, but also positive  $d$-dimensional  Lebesgue  measure.  Interestingly, there is no implication in the other direction.  More precisely, there exist compact subsets of  $\mathbb{R}^d$ with positive Lebesgue measure but with Fourier dimension equal to 0, see \cite[Example 18]{ekstrom}.  In particular,  none of the Fourier dimension estimates we provide in this paper are  implied by statements about positivity of Lebesgue measure.

We metrise $G(d,k)$ using the Hausdorff distance. The Hausdorff distance between two non-empty compact  subsets $A,B$ of a compact metric space $(X,d)$ is given by
\[d_H(A,B) = \max\left\{\sup_{a\in A}\inf_{b \in B}d(a,b), \ \sup_{b\in B} \inf_{a \in A}d(b,a)\right\}.\]
 The metric we use on $G(d,k)$ is then
\begin{align}
    d_{(d,k)}(s,t) = d_H(s\cap S ^{d-1}, t\cap S ^{d-1})\label{d-d-k-metric}
\end{align}
for $s,t \in G(d,k)$ and where $ S ^{d-1}$ is the unit sphere in $\mathbb{R}^d$.

\section{Results and applications}

We state our main result in terms of an abstract scaling condition, which we refer to as $\beta$-scaling.  We then apply this result in two different ways to obtain less abstract corollaries.   Given $\xi \in \mathbb{R}^d \setminus\{0\}$ and $\eta>0$, let
\[
S_{ \xi , \eta}=\{s\in G(d,k) : |\xi\cdot x_i^s|<\eta|\xi|,  \ \forall\, i = 1, \dots, k\}.
\]
Essentially, $s \in S_{\xi, \eta}$ if it is in an $\eta$-neighbourhood of the orthogonal complement of $\xi$.  We say $\Gamma \subseteq G(d,k)$ is $\beta$\emph{-scaling} for $\beta \geq 0$ if it supports a Borel measure $\gamma$ such that 
\[
\gamma(S_{ \xi, \eta}) \lesssim_\beta \eta^\beta
\]
for all $\xi \in \mathbb{R}^d \setminus\{0\}$ and $\eta>0$.

\begin{theorem}\label{grassmannian-generalisation}
Let $\Gamma\subseteq G(d,k)$ and assume $\Gamma$ is $\beta$-scaling.  Let $E\subseteq\mathbb{R}^d$ be a compact $(d,k,\Gamma)$-set.  Then
\[\dimf E \ge \min\{2\beta,d\}.\]
Moreover, if $\beta>d/2$, then $E$ has positive $d$-dimensional Lebesgue measure.  
\end{theorem}

We defer the proof of Theorem \ref{grassmannian-generalisation} to Section \ref{proof1}. Our first application of Theorem \ref{grassmannian-generalisation} is to bound the Fourier dimension of $(d,k,\Gamma)$-sets from below in terms of the Hausdorff dimension of $\Gamma$. The next result guarantees that $\Gamma$ is $\beta$-scaling if its Hausdorff dimension is sufficiently large.

\begin{proposition} \label{scaling}
Let $\Gamma\subseteq G(d,k)$ and suppose $k(d-1-k)<b<\dimh \Gamma$.  Then $\Gamma$ is $(b-k(d-1-k))$-scaling. 
\end{proposition}

We defer the proof of Theorem \ref{scaling} to Section \ref{proof2}. Combining Theorem \ref{grassmannian-generalisation} and Proposition \ref{scaling} we immediately get the following result.

\begin{corollary}\label{dimensionthm}
Let $\Gamma\subseteq G(d,k)$ and  $E\subseteq\mathbb{R}^d$ be a compact $(d,k,\Gamma)$-set.  If $2(\dimh \Gamma -k(d-1-k)) \leq d$, then 
\[\dimf E \ge  2(\dimh \Gamma -k(d-1-k)).\]
Otherwise, if $2(\dimh \Gamma -k(d-1-k)) > d$, then $\dimf E=d$ and  $E$ has positive $d$-dimensional Lebesgue measure.  
\end{corollary}

The special case of Corollary \ref{dimensionthm} when $d=2>1=k$ was proved in \cite{oberlin}.  The final conclusion of Corollary \ref{dimensionthm} giving conditions guaranteeing positive measure can be deduced from \cite[Theorem 1.3]{oberlin_israel}, but our Fourier dimension estimates are new to the best of our knowledge.   Corollary \ref{dimensionthm} is sharp in the sense that for arbitrary $d>k$ there exist $(d,k,\Gamma)$-sets $E$ with $\dimh \Gamma = k(d-1-k)$ and $\dimf E = 0$.  Such sets can be constructed by choosing $V \in G(d,d-1)$ and then  $\Gamma$  to be all $k$-dimensional subspaces of $\mathbb{R}^d$ which lie in $V$.  With this choice of $\Gamma$ in place, $V \cap B(0,1)$  is a compact $(d,k,\Gamma)$-set and has  Fourier dimension 0 since any subset of a hyperplane always has  Fourier dimension 0.

Specialising to $(d,k)$-sets, we recover Falconer's result for $k>d/2$ and obtain new Fourier dimension bounds.

\begin{corollary}\label{dimensionthm2}
Let   $E\subseteq\mathbb{R}^d$ be a compact $(d,k)$-set.  If $k  \leq d/2$, then 
\[\dimf E \ge  2k.\]
Otherwise, if $k> d/2$, then $\dimf E=d$ and $E$ has positive $d$-dimensional Lebesgue measure.  
\end{corollary}

To motivate these results, briefly consider compact $(3,2)$-sets which are known to have positive 3-dimensional Lebesgue measure.   Corollary \ref{dimensionthm} provides a strengthening of this result by giving the same conclusion with only the requirement that the set of orientations has Hausdorff dimension $>3/2$. Recall that the dimension of the set of available orientations is 2.   However, a further improvement is possible using \cite[Theorem 1.3]{oberlin_israel}.  From this result it may be shown that the conclusion of positive measure can be reached with a set of orientations of dimensions $>1$.  Moreover, this is sharp since if $\dimh \Gamma = 1$, then  $(3,2, \Gamma)$-sets can have zero 3-dimensional Lebesgue measure.  For example, consider the product of a Besicovitch set $E' \in \mathbb{R}^2$  (that is, a Kakeya set with zero 2-dimensional measure) with a unit line segment.  By disintegration of 3-dimensional Lebesgue measure we see that $E=E' \times [0,1]$ is a null set and, moreover, is a $(3,2, \Gamma)$ with $\dimh \Gamma = 1$.  Higher dimensional versions of this construction are also possible but left to the reader.

Next we point out that Theorem \ref{grassmannian-generalisation} often gives an improvement over  Corollary \ref{dimensionthm} if there is additional geometric information known about $\Gamma$.  There are many examples possible here and so we just highlight some of our favourites.  Generally, better estimates will be possible for sets  $\Gamma$ which stay sufficiently far away from hyperplanes.

Let $d \geq 3$ and identify $G(d,1)$ and $S^{d-1}$ in the natural way.  We say a set $\Gamma \subseteq G(d,1)$ is a \emph{non-degenerate sphere} if it is a $(d-2)$-dimensional sphere with diameter strictly less than 2, that is, it is not the intersection of $S^{d-1}$ with a hyperplane. 

\begin{proposition} \label{spheres}
Let $\Gamma \subseteq G(d,1)$ be a \emph{non-degenerate sphere}.   If $d=3$, then $\Gamma$ is $1/2$-scaling. If $d \geq 4$, then $\Gamma$ is $1$-scaling.
\end{proposition}

We defer the proof of Proposition  \ref{spheres} to Section \ref{proof3}. As a consequence, we get the following result in the setting of restricted Kakeya sets.  This is a strict improvement over Corollary \ref{dimensionthm} which does not give any non-trivial lower bounds in this setting.

\begin{corollary}\label{restricted} 
Let $E \subseteq \mathbb{R}^d$ be a compact set containing a unit line segment in a non-degenerate sphere of directions.  If $ d =3$, then
\[
 \dimf E \geq 1.
\]
If $ d \geq 4$, then
\[
 \dimf E \geq 2.
\]
\end{corollary}

Corollary \ref{restricted} is sharp  for $d=3,4$.  This can be seen by following result about cones, which may be of interest in its own right. 

\begin{thm} \label{cones}
For $d \geq 1$, the cone
\[
\mathcal{C}^d =\left\{ (\xi_1, \dotsc, \xi_d, \xi_{d+1}) \in \mathbb{R}^{d+1}:  \lvert (\xi_1, \dotsc, \xi_d ) \rvert = \vert \xi_{d+1} \rvert \right\} 
\]
in $\mathbb{R}^{d+1}$ has Fourier dimension $d-1$.
\end{thm}
We defer the proof of Theorem  \ref{cones} to Section \ref{proof4}.

It is perhaps noteworthy that in the case $d \geq 4$ in the above, we get the same lower bound for the Fourier dimension as that for genuine Kakeya sets, where lines in every direction are present, not just in a non-degenerate sphere of directions, see Corollary \ref{dimensionthm2}. Moreover, the non-degeneracy condition is necessary in the above, since a hyperplane contains a unit line segment in a $(d-2)$-dimensional sphere of directions, but has Fourier dimension 0.

We note that Corollary \ref{restricted} gives non-trivial lower bounds for the Hausdorff dimension of restricted Kakeya sets $E$ in the case $d \geq 4$.  However, these bounds can be improved in all dimensions by noting that the orthogonal projection of such $E$ onto the hyperplane orthogonal to the subspace spanned by the centre of $\Gamma$ is a genuine Kakeya set living in ambient dimension $d-1$.  Then, since Hausdorff dimension cannot increase under projection, one can bound $\dimh E$ from below by applying the state of the art estimates for the Kakeya problem.  We observe that Fourier dimension can (rather easily) increase under projection, and so such a reduction is not possible for Fourier dimension.

\section{Proof of Theorem \ref{grassmannian-generalisation}} \label{proof1}

Write  $C_c^\infty(\mathbb{R})$ for the space of infinitely differentiable functions with compact support. Let  $\phi \in C_c^\infty(\mathbb{R})$ be nonnegative with support  $\spt\, \phi \subseteq [0,1]$ and such that $\int_0^1\phi(x)\dd x=1$. Observe that the Fourier transform of $\phi$ is bounded above by $1$ since 
\[|\hat{\phi}(\xi)|\le\left|\int_{0}^{1}\phi(x)e^{-2\pi i \xi x}\dd x \right|\le \int_{0}^{1}\phi(x)\left|e^{-2\pi i \xi x}\right|\dd x =1.\]
Moreover,  $\phi$ is an element of the Schwartz space $\mathcal{S}(\mathbb{R})$, and so $|\hat{\phi}(\xi)|$ decays rapidly as $|\xi| \to \infty$, see \cite[Chapter 3]{mattila_2015}. The following proof broadly follows Oberlin \cite[Proposition 2]{oberlin}, see also the exposition \cite[Theorem 11.3]{mattila_2015}.  Our main new idea is that, when considering balls in $\mathbb{R}^d$ instead of lines in $\mathbb{R}^2$, one needs more understanding of the geometry of $\Gamma$, which is achieved via $\beta$-scaling.  Moreover, we use a natural extension of Oberlin's argument to higher dimensions which requires multiple uses of the Schwartz function $\phi$.

\begin{proof}[Proof of Theorem \ref{grassmannian-generalisation}]
Let $\gamma$ be the Borel measure supported on $E$ coming from the definition of $\beta$-scaling.  Assume for now that the map $s\mapsto t_s$ coming from the definition of $E$ is measurable. We address the issue of measurability  at the end of the proof using a standard discretisation approach.  Since $s\mapsto t_s$ is measurable we can use the Riesz representation theorem, see \cite[2.14 Theorem]{rudin}, to define a measure $\mu\in\mathcal{M}(E)$ by
\begin{align}\int_{E}f\dd \mu=\int_{G(d,k)}\int_{[0,1]^k}f\left(t_s+\sum_{i=1}^kr_ix_i^s\right)\phi(r_1)\cdots\phi(r_k)\dd r \dd \gamma(s)\label{mu-integral-general}\end{align}
for continuous functions $f$ on $\mathbb{R}^d$. 

Let $\xi\in\mathbb{R}^d$ with $|\xi|>1$. The Fourier transform of $\mu$ at $\xi\in\mathbb{R}^d$ is given by
\begin{align*}
    \hat{\mu}(\xi) &= \int_{G(d,k)}\int_{[0,1]^k}e^{-2\pi i \left(t_s+\sum_{i=1}^kr_ix_i^s\right) \cdot \xi}\phi(r_1)\cdots\phi(r_k)\dd r\dd \gamma(s).
\end{align*}
First we integrate out the translations by
\[\int_{[0,1]^k}e^{-2\pi i \left(t_s+\sum_{i=1}^kr_ix_i^s\right) \cdot \xi}\phi(r_1)\cdots\phi(r_k)\dd r = e^{-2\pi i t_s\cdot \xi} \prod_{i=1}^{k}\int_0^1e^{-2\pi i r_i x_i^s\cdot \xi}\phi(r_i)\dd r_i.\]
Therefore \[\left|e^{-2\pi i t_s\cdot \xi} \prod_{i=1}^{k}\int_0^1e^{-2\pi i r_i x_i^s\cdot \xi}\phi(r_i)\dd r_i\right| = \prod_{i=1}^{k} |\hat{\phi}(\xi\cdot x_i^s)|,\] and 
\[|\hat{\mu}(\xi)| \le \int_{G(d,k)} \prod_{i=1}^{k} |\hat{\phi}(\xi\cdot x_i^s)|\dd \gamma(s).\]

We  split the above integral into two pieces which are then bounded separately. Let $\eta>0$ and recall  the sets
\[
S_{\xi,\eta}=\{s\in G(d,k) : |\xi\cdot x_i^s|<\eta|\xi|, \ \forall\, i = 1, \dots, k\}
\]
used to define $\beta$-scaling.   Notice that if $s\not\in S_{\xi,\eta}$ then there is some $j\in\{1,\ldots,k\}$ such that $|\xi\cdot x_j^s|\ge \eta |\xi|$. We then use the fact that $\hat{\phi}$ is bounded above by $1$ and rapidly decreasing to conclude that, for any $N>1$, there exists a constant $C_N$ such that
\[
|\hat{\phi}(\xi\cdot x_j^s)| \le \frac{C_N}{|\xi\cdot x_j^s|^N}.
\]
 We then have
\begin{align*}
    \int_{G(d,k)\backslash S_{\xi,\eta}} \ \prod_{i=1}^{k} |\hat{\phi}(\xi\cdot x_i^s)| \dd \gamma(s) &\le \int_{G(d,k)\backslash S_{\xi,\eta}}  |\hat{\phi}(\xi\cdot x_j^s)|  \dd \gamma(s)\\
    &\le \int_{G(d,k)\backslash S_{\xi,\eta}}  \frac{C_N}{|\xi\cdot x_j^s|^N} \dd \gamma(s)\\
    &\lesssim_N (\eta|\xi|)^{-N}.
\end{align*}
On the other hand, we have the simple estimate 
\begin{align*}
    \int_{S_{\xi,\eta}} \  \prod_{i=1}^{k} |\hat{\phi}(\xi\cdot x_i^s)| \dd \gamma(s) &\le \int_{S_{\xi,\eta}}  \dd \gamma(s) = \gamma(S_{\xi,\eta}) \lesssim \eta^\beta
\end{align*}
where the final inequality is the only place where we use the $\beta$-scaling property.  This gives
\begin{align*}|\hat{\mu}(\xi)|\le \int_{G(d, k)} \  \prod_{i=1}^{k} |\hat{\phi}(\xi\cdot x_i^s)| \dd \gamma(s) & \lesssim_N  \eta^\beta + (\eta|\xi|)^{-N}. \end{align*}
Now let $0<\alpha<1$ and set $\eta=|\xi|^{-\alpha}$. Then, setting $N=\frac{\alpha \beta}{1-\alpha}$ yields
\[|\hat{\mu}(\xi)|\lesssim_{\alpha} |\xi|^{-\alpha \beta}.\]
This proves $\dimf E\ge \min\{2\alpha \beta,d\}$ and letting $\alpha\rightarrow 1$ gives $\dimf E\ge \min\{2\beta,d\}$ as required. Moreover, if $\beta>d/2$, $E$ has positive $d$-dimensional Lebesgue measure.

It remains to address the measurability issue mentioned earlier. If $s \mapsto t_s$ is not measurable then we discretise $\mu$ as follows.  Let $\{z_1,z_2,\ldots,z_m\}$ be a maximal $(1/n)$-separated set of points in $\Gamma$ and define a measure $\gamma_n$ by
\[\gamma_n = C_{\gamma,n}\sum_{i=1}^m\gamma\Big(B\big(z_i,1/n \big)\Big)\delta_{z_i}\]
where $C_{\gamma,n}$ is a normalisation constant chosen such that $\gamma_n(G(d,k))=\gamma(G(d,k))$ and $\delta_{z_i}$ is a Dirac mass at $z_i$.  Then we may define  measures $\mu_n$ by replacing $\gamma$ with $\gamma_n$ in (\ref{mu-integral-general}). The  argument given above shows that $|\hat{\mu}_n(\xi)|\lesssim_{\alpha} |\xi|^{-\alpha\beta}$ for $n$ sufficiently large.  Moreover, $(\mu_n)_n$ converges weakly to a measure $\mu\in\mathcal{M}(E)$, and the proof above goes through with this measure.
\end{proof}

\section{Proof of Proposition \ref{scaling}} \label{proof2}

For $0<b<\dimh \Gamma$,  Frostman's  lemma, see  \cite[8.17. Theorem]{mattila_1995},  guarantees  the existence of a compactly supported measure $\gamma \in \mathcal{M}(\Gamma)$ such that, \begin{align}\label{frostman-b-condition}
\gamma(B(s,r))\lesssim_b  r^b
\end{align}
for all $s\in G(d,k)$ and $r>0$. Straight from the definition of $S_{ \xi ,\eta}$, there exists a  constant $C \geq 1$ such that  $S_{ \xi ,\eta} \subseteq T_{ \xi ,\eta}$ where
\[ T_{ \xi ,\eta}:= \{s\in G(d,k): d_{(d,k)}(s,t) < C\eta\text{ for some $t\in G(d,k)$ contained in $ \xi ^\perp$}\}\]
where $d_{(d,k)}$ is the metric defined in (\ref{d-d-k-metric}). 

Let $G_{ \xi }$ denote the set of $k$-dimensional subspaces of $ \xi ^\perp$  and note that $G_{ \xi }$ is a Grassmannian manifold isomorphic to $G(d-1,k)$. From the definition of the metric $d_{(d,k)}$, we have for $t,t'\in G_ \xi $ that $d_{(d,k)}(t,t')=d_{(d-1,k)}(t,t')$, i.e. it is safe to use the metric $d_{(d,k)}$ on $G_ \xi $. 

Let \[\{B(t_i,\eta)\cap G_{ \xi }:t_i\in G_ \xi , 1\le i\le N_\eta\}\] 
be a minimal covering of $G_ \xi $ by $\eta$-balls noting that
\begin{align} \label{coveringN}
N_\eta\lesssim \eta^{-k(d-1-k)}
\end{align}
where the exponent $k(d-1-k)$ comes from the dimension of $G_ \xi $.  For $s\in T_{ \xi ,\eta}$  there exists  $t\in G_ \xi $ with $d_{(d,k)}(s,t)<C\eta$ and  $t_i$, the centre of a ball in the covering, such that $d_{(d,k)}(t,t_i)<\eta$. Therefore
\begin{align}
    T_{ \xi ,\eta} \subseteq \bigcup_{i=1}^{N_\eta} B(t_i,(C+1)\eta)\label{t-xi-eta}
\end{align}
allowing us to relate coverings of $G_ \xi $ to coverings of $T_{ \xi ,\eta}$. Using  (\ref{frostman-b-condition}), (\ref{coveringN}) and  (\ref{t-xi-eta}) we get
\begin{align*}
  \gamma(S_{ \xi ,\eta})  \leq \gamma(T_{ \xi , \eta})\le \sum_{i=1}^{N_\eta} \gamma(B(t_i,(C+1)\eta))
    \lesssim_b \eta^{-k(d-1-k)} \eta^{b}
    = \eta^{b-k(d-1-k)}
\end{align*}
completing the proof.

\section{Proof of Proposition  \ref{spheres}} \label{proof3}
Let $\gamma$ be the normalised spherical measure on $\Gamma$.  Let $\xi \in \mathbb{R}^d \setminus\{0\}$ and $\eta>0$. We may assume that $\eta$ is much smaller than the diameter of $\Gamma$. Estimating $\gamma(S_{ \xi, \eta})$  immediately reduces to estimating  the $\gamma$-volume of the intersection of the sphere $\Gamma$ with the $\eta$-neighbourhood of a plane of the same dimension.  This in turn reduces to understanding the intersection of spheres and planes. There are two types of such intersection: tangential and non-tangential.   We need a parameter $r$  to make this distinction more quantitative.  In the non-tangential case, the intersection of a $(d-2)$-dimensional sphere and a plane of the same dimension is itself a sphere of dimension $(d-3)$. (In the case $d=3$ we think of a 0-dimensional sphere as  two distinct  points with centre given by the midpoint.)   Let $r>0$ be the minimal distance from the centre of the intersection to the  original sphere $\Gamma$.  The case $r=0$ is then the tangential case.  If $r \leq 3 \eta$, then $S_{ \xi, \eta}$ is contained in  a ball of radius $\lesssim \sqrt{\eta}$, where the square root comes from the basic geometry of a sphere.  Therefore
\[
\gamma(S_{ \xi, \eta}) \lesssim \sqrt{\eta}^{(d-2)} = \eta^{d/2-1}.
\]
On the other hand, if $r > 3 \eta$, then  $S_{ \xi, \eta}$ is contained in an $\eta/\sqrt{r}$-thickening of a sphere of dimension $(d-3)$ and diameter $\lesssim \sqrt{r}$.  This can be covered by
\[
\lesssim \left(\frac{\sqrt{r}}{\eta}\right)^{d-3} \frac{\eta/\sqrt{r}}{\eta}  = \eta^{3-d}\,  r^{d/2-2}
\]
many $\eta$-balls each of $\gamma$ measure $\lesssim \eta^{d-2}$.  Therefore,
\[
\gamma(S_{ \xi, \eta}) \lesssim  \eta \, r^{d/2-2} \lesssim  \eta^{\min\{1, d/2-1\}}.
\]
Therefore,   $\Gamma$ is $\min\{1, d/2-1\}$-scaling, proving the claim.

\section{Proof of Theorem  \ref{cones}} \label{proof4}

 The case $d =1$ is trivial, so assume that $d \geq 2$. The lower bound $\dimf  \mathcal{C}^d  \geq d-1$ follows (for example) by using \eqref{sphereasymp} below and considering the measure defined by 
\[ f \mapsto \int_{\mathbb{R}} \psi(r) \int_{S^{d-1}} f(r x, r) \, d\sigma(x) \, dr, \]
for any non-negative Borel function $f$, where $\sigma$ is the rotation invariant Borel probability measure on $S^{d-1}$, and $\psi$ is a bump function on $[1,2]$ with $\int \psi = 1$. 

Suppose for a contradiction that $\dimf  \mathcal{C}^d  > d-1$. Then 
there exists $\alpha >d-1$ 
 and a Borel probability measure $\mu$ on $ \mathcal{C}^d $, such that 
\begin{equation} \label{mudecay}  \lvert\widehat{\mu}(\xi) \rvert  \lesssim \lvert\xi\rvert^{-\alpha/2} \qquad\forall \xi \in \mathbb{R}^{d+1}. \end{equation}
By symmetry, and by replacing $\mu$ with $f \mu$ for an appropriate bump function $f$ (see \cite[Lemma~1]{ekstrom}), it may be assumed that for some $\epsilon>0$, 
\begin{equation} \label{sptcdn} \supp \mu \subseteq \{ (\xi, \lvert\xi\rvert ) \in \mathbb{R}^d \times \mathbb{R}: \epsilon \leq \lvert\xi\rvert  \leq 1/\epsilon \}. \end{equation}
 Let $\nu$ be the Borel probability measure on $ \mathcal{C}^d $ defined by
\begin{equation} \label{nudefn} \int f \, d\nu = \int_{\mathbb{R}^d \times  \mathbb{R}}  \int_{S^{d-1}} f(\lvert x\rvert w, z) \, d\sigma(w) \, d\mu(x, z), \end{equation}
for any non-negative Borel function $f$. Then 
\begin{align*} \widehat{\nu}(\xi) &= \int_{\mathbb{R}^d \times \mathbb{R}} \int_{S^{d-1}} e^{-2\pi i \langle \xi, (\lvert x\rvert w  , z) \rangle } \, d\sigma(w) \, d\mu(x,z)   \\
&=  \int_{\mathbb{R}^d \times \mathbb{R}} \int_{O(d)} e^{-2\pi i \langle \xi, (Ux , z) \rangle } \, d\lambda(U) \, d\mu(x,z) \\
&= \int_{O(d)} \widehat{ \mu} (U^*(\xi_1,\dotsc ,\xi_d), \xi_{d+1} ) \, d\lambda(U),  \end{align*} 
where $\lambda$ is the Haar probability measure on $O(d)$. Hence $\nu$ satisfies
\begin{equation} \label{nudecay} \lvert \widehat{\nu} (\xi)\rvert   + \left\lvert \widehat{\widetilde{\nu}} (\xi) \right\rvert\lesssim \lvert \xi\rvert^{-\alpha/2} \qquad\forall \xi \in \mathbb{R}^{d+1}, \end{equation}
where $\widetilde{\nu}$ is the pushforward of $\nu$ under $(x_1, \dotsc, x_d, x_{d+1}) \mapsto (x_1, \dotsc, x_d, -x_{d+1})$.  Let $\pi: \mathbb{R}^{d+1} \to \mathbb{R}$ be the map $(x_1,\dotsc,x_d,x_{d+1}) \mapsto x_{d+1}$. Since $\supp \mu \subseteq  \mathcal{C}^d $, and by \eqref{sptcdn}, the formula \eqref{nudefn} can also be written as 
\[ \int f \, d\nu = \int_{\mathbb{R}} \int_{S^{d-1}} f(zw, z) \, d\sigma(w) \, d\pi_{\#}\mu(z). \] 
Hence another expression for $\widehat{\nu}$ is
\[ \widehat{\nu}(\xi) = \int e^{-2\pi i z \xi_{d+1} } \widehat{\sigma}(z (\xi_1, \dotsc, \xi_d)) \, d\pi_{\#}\mu(z). \]
Let $\omega_{d-1}$ be the surface area of $S^{d-1}$. The asymptotic formula (see \cite[Appendix~B]{grafakos})
\begin{equation} \label{sphereasymp} \widehat{\sigma}(\xi) = \frac{2}{\omega_{d-1}}\lvert \xi\rvert^{-(d-1)/2} \cos\left(2 \pi \lvert \xi\rvert  - \frac{\pi(d-1)}{4}  \right) + O\left(\lvert \xi\rvert^{-(d+1)/2}\right), \end{equation}
gives, by taking $\xi_{d+1} = \lvert (\xi_1, \dotsc, \xi_d)\rvert$, 
\begin{multline*}  e^{ \frac{i\pi(d-1)}{4} } \widehat{\nu}(\xi, \lvert \xi\rvert ) +  e^{ \frac{-i\pi(d-1)}{4} }\widehat{ \widetilde{\nu} }(\xi, \lvert \xi\rvert )  \\
= \frac{4}{\omega_{d-1}}  \lvert \xi\rvert^{-(d-1)/2} \int z^{-(d-1)/2} \cos^2\left( 2 \pi \lvert \xi\rvert  z  - \frac{\pi(d-1)}{4} \right) \, d\pi_{\#}\mu(z) \\ +O\left(\lvert \xi\rvert^{-(d+1)/2}\right). \end{multline*}
Comparing \eqref{nudecay} to the above will give a contradiction, by the following identity:
\begin{equation} \label{halfidentity} \lim_{r \to \infty} \int \cos^2\left(r z+t \right) \, d\pi_{\#}\mu(z) = 1/2 \qquad\forall t \in \mathbb{R}.  \end{equation}
It remains to prove \eqref{halfidentity}. Since $d \geq 2$ and $\alpha/2 > (d-1)/2 \geq 1/2$, condition \eqref{mudecay} gives $\pi_{\#} \mu \in L^2(\mathbb{R})$ (see e.g.~\cite[Theorem~3.3]{mattila_2015}). Hence $\pi_{\#} \mu \in L^1(\mathbb{R})$ with $\lVert \pi_{\#} \mu \rVert_1=1$, and \eqref{halfidentity} then follows by approximating $\pi_{\#} \mu$ in $L^1$ with a finite linear combination of characteristic functions of disjoint intervals.

\section{Further work and some questions}

The simplicity of Oberlin's argument in \cite{oberlin} makes it very appealing to try to adapt it to a range of different problems, such as the ones we consider here.  Another problem is the following dual to the Kakeya problem, motivated by work of Wolff \cite{wolff}.  Suppose $E \subseteq \mathbb{R}^d$ contains a sphere of every radius $r \in (0,1)$.  Wolff \cite{wolff} proved that $E$ necessarily has Hausdorff dimension $d$.  This result was proved by Kolasa and Wolff for $d \geq 3$ \cite{wolff1}, which is much easier than the $d=2$ case.  Here there is a trivial lower bound of $d-1$ for the Fourier dimension, since a single sphere in $\mathbb{R}^d$ has Fourier dimension $d-1$, see \cite[Equation (3.42)]{mattila_2015}.  For Kakeya sets there is no non-trivial bound since line segments have Fourier dimension 0 for $d \geq 2$.  Despite non-trivial estimates existing for the Fourier dimension of Kakeya sets, we are unaware of any improvement over the trivial lower bound for the dual problem and pose this as a question.

\begin{ques} \label{ques}
If $E \subseteq \mathbb{R}^d$ contains a sphere of every radius $r \in (0,1)$, then is it true that $\dimf E = d$?
\end{ques}

For $E$ as in Question \ref{ques},  for every $r \in (0,1)$ there is a centre $x_r$ such that
\[
x_r+r \theta \in E
\]
for all $\theta \in S^{d-1}$. One can try to adapt Oberlin's argument by defining a  measure $\mu$ on $E$ as in \eqref{mu-integral-general}  by, for example, 
\[
\int_{E}f\dd \mu=\int_{0}^1\int_{S^{d-1}}f\left(x_r+ r\theta\right)\phi(r) \dd \sigma^{d-1}\theta \dd r
\]
where $\sigma^{d-1}$ is the spherical measure on $S^{d-1}$.  Using Fubini's theorem the Fourier transform of $\mu$ at $\xi\in\mathbb{R}^d$ is then
\begin{align*}
    \hat{\mu}(\xi) &=\int_{S^{d-1}}\int_{0}^1e^{-2\pi i \left(x_r+ r\theta\right) \cdot \xi}\phi(r)\dd r  \dd \sigma^{d-1}\theta. 
\end{align*}
The problem now comes that the centre $x_r$ depends on $r$ (not on $\theta$ for example) and so cannot be integrated out to reduce the problem to studying the Fourier transform of $\phi$.  Alternatively, one could try to associate the Schwartz function $\phi$ with the parameter $\theta$, but this leads to a very awkward integral.

Another very intriguing problem is, of course, the Fourier analytic formulation of the Kakeya problem.

\begin{ques} \label{kakeyaques}
If $E \subseteq \mathbb{R}^d$ contains a unit line segment in every direction, then is it true that $\dimf E = d$?
\end{ques}

We are not aware of any improvements over $\dimf E \geq  2$ for Kakeya sets $E$.  This estimate  follows from Corollary \ref{dimensionthm2}.  Oberlin only gives this result for $d=2$ and the improvement we obtain requires the additional geometric argument used in proving Proposition \ref{scaling}.  We have some doubts about a positive answer to Question \ref{kakeyaques} in general.

The measures $\mu$ we use in this paper  do not take into account the specific placement of the lines (or $k$-dimensional balls); indeed, when bounding the Fourier transform of $\mu$ the dependency on the translation $t_s$ is simply removed by making a trivial estimate.  It would be interesting to try to modify the measure to account for translations, but we do not know how to do this in an effective way.

\section*{Acknowledgements}

We thank Pertti Mattila for helpful discussions, especially for pointing out the  papers  \cite{ekstrom, oberlin_israel}. We also thank Pablo Shmerkin for helpful discussions and Tuomas Orponen for pointing out some references.

\end{document}